\documentclass[12pt,twoside]{amsart}
\usepackage{amssymb}
\usepackage{amscd}

\title[Effective freeness]{Effective base point free 
theorem for log canonical pairs---Koll\'ar type theorem}
\author{Osamu Fujino} 
\subjclass[2000]{Primary 14C20; Secondary 14E30.}
\date{2009/6/20}
\address{Department of Mathematics\endgraf 
Faculty of Science\endgraf 
Kyoto University\endgraf
Kyoto 606-8502\endgraf  
Japan}
\email{fujino@math.kyoto-u.ac.jp}
\newcommand{\Supp}[0]{{\operatorname{Supp}}}
\newcommand{\Bs}[0]{{\operatorname{Bs}}}
\newcommand{\codim}[0]{{\operatorname{codim}}}
\newcommand{\Exc}[0]{{\operatorname{Exc}}}
\newtheorem{thm}{Theorem}[section]

\newtheorem{thm-a}{Theorem}[subsection]
\newtheorem{lem-a}[thm-a]{Lemma}
\newtheorem{cor-a}[thm-a]{Corollary}

\theoremstyle{definition}

\newtheorem{rem}[thm]{Remark}
\newtheorem*{ack}{Acknowledgments}      
\newtheorem*{notation}{Notation}         
\newtheorem{say}[thm]{}

\newtheorem{say-a}[thm-a]{}

\begin{document}
\bibliographystyle{amsalpha+}

\maketitle

\begin{abstract}
Koll\'ar's effective base point free theorem 
for kawamata log terminal 
pairs is very important and was used in Hacon--McKernan's 
proof of pl flips. 
In this paper, we generalize Koll\'ar's 
theorem for {\em{log canonical}} pairs. 
\end{abstract} 


\section{Introduction}
The main purpose of this paper 
is to show the power of the new cohomological 
technique introduced in \cite{ambro}. 
The following theorem is the main theorem of 
this short note. 
It is a generalization of \cite[1.1 Theorem]{kollar}. 
Koll\'ar proved it only for kawamata log terminal pairs. 

\begin{thm}[Effective base point free theorem]\label{main}
Let $(X, \Delta)$ be a {\em{projective}} log canonical 
pair with $\dim 
X=n$. Note that $\Delta$ is an effective $\mathbb Q$-divisor on $X$. 
Let $L$ be a nef Cartier divisor on $X$. Assume that 
$aL-(K_X+\Delta)$ is nef and log big for some $a\geq 0$. 
Then there exists a positive integer $m=m(n, a)$,  
which only depends on $n$ and $a$,  such that 
$|mL|$ is base point free.  
\end{thm} 

For the relative statement, see Theorem \ref{277} below. 

\begin{rem}
We can take $m(n, a)=2^{n+1}(n+1)!(\ulcorner a\urcorner+n)$ 
in Theorem \ref{main}. 
\end{rem}

By the results in \cite{ambro}, 
we can apply a modified 
version of X-method to log canonical pairs. 
More precisely, generalizations of 
Koll\'ar's 
vanishing and torsion-free theorems 
to the context of embedded simple normal crossing pairs replace the 
Kawamata--Viehweg 
vanishing theorem in the world of log canonical pairs. 
For the details, see \cite{fujino4}. 
Here, we generalize Koll\'ar's arguments in \cite{kollar} 
for log canonical pairs. 
This further illustrates the usefulness 
of our new cohomological 
package. 
For the benefit of the reader, we will explain the new vanishing 
and torsion-free theorems in the appendix 
(see Section \ref{sec3}). 
The starting point of our main theorem is 
the next theorem (see \cite[Theorem 7.2]{ambro}). 
For the proof, see \cite[Theorem 4.4]{fujino4}. 
Ambro's original statement is much more general than 
Theorem \ref{bpf}. Unfortunately, 
he gave no proofs in \cite{ambro}. 

\begin{thm}[Base point free theorem for log canonical pairs]\label{bpf}
Let $(X, \Delta)$ be a log canonical pair and $L$ a $\pi$-nef 
Cartier divisor on $X$, where $\pi:X\to V$ is a {\em{projective}} 
morphism. Assume that $aL-(K_X+\Delta)$ is $\pi$-nef and 
$\pi$-log big for some 
positive real number $a$. 
Then $\mathcal O_X(mL)$ is $\pi$-generated for 
all $m\gg 0$. 
\end{thm}

The paper \cite{fu4} may help the reader to understand 
Theorem \ref{bpf}. In \cite{fu4}, Theorem \ref{bpf} 
is proved under the assumption that 
$V$ is a point and $aL-(K_X+\Delta)$ is ample. 

We summarize the contents of this paper. 
In Section \ref{sec2}, we prove 
Theorem \ref{main}. 
In Subsection \ref{2.1}, 
we give a slight generalization of 
Koll\'ar's modified base point freeness method. 
We change Koll\'ar's formulation 
so that we can apply our new cohomological 
technique. 
In Subsection \ref{sub2.2}, we use the modified 
base point freeness method to obtain Theorem \ref{main}. 
Here, we need Theorem \ref{bpf}. 
Section \ref{sec3} is an appendix, where 
we quickly review 
our new vanishing and torsion-free theorems 
for the reader's convenience. 
The reader can find Angehrn--Siu type effective base point 
freeness and point separation for {\em{log canonical}} 
pairs in \cite{fujino3}. 

\begin{notation} 
We will work over the complex number field $\mathbb C$ throughout  
this paper. 

Let $r$ be a real number. 
The {\em{integral part}} $\llcorner r\lrcorner$ is 
the largest integer at most $r$ and 
the {\em{fractional part}} $\{r\}$ is defined by $r-\llcorner r\lrcorner$. 
We put $\ulcorner r\urcorner =-\llcorner -r \lrcorner$ and 
call it the {\em{round-up}} of $r$. 

Let $X$ be a normal variety and $B$ an effective 
$\mathbb Q$-divisor such that 
$K_X+B$ is $\mathbb Q$-Cartier. 
Then we can define the discrepancy 
$a(E, X, B)\in \mathbb Q$ for 
every prime divisor $E$ over $X$. 
If $a(E, X, B)\geq -1$ (resp.~$>-1$) for 
every $E$, then $(X, B)$ is called {\em{log canonical}} (resp.~{\em{kawamata 
log terminal}}). We sometimes abbreviate log canonical 
to {\em{lc}}. 

Assume that $(X, B)$ is log canonical. 
If $E$ is a prime divisor over $X$ 
such that 
$a(E, X, B)=-1$, 
then $c_X(E)$ is called a {\em{log canonical center}} 
({\em{lc center}}, for short) 
of $(X, B)$, where $c_X(E)$ is 
the closure of the 
image 
of $E$ on $X$. 
A $\mathbb Q$-Cartier $\mathbb Q$-divisor $L$ 
on $X$ is called {\em{nef and log big}} 
if $L$ is nef and big and $L|_{W}$ is big for every 
lc center $W$ of $(X, B)$. 
The relative version of nef and log bigness can be defined 
similarly. 

For a $\mathbb Q$-divisor $D=\sum _{i=1}^r d_i D_i$, 
where $D_i$ is a prime divisor for every $i$ and $D_i\ne D_j$ for 
$i\ne j$, we call $D$ a {\em{boundary}} $\mathbb Q$-divisor 
if $0\leq d_i \leq 1$ for every $i$. 
We denote by $\sim_{\mathbb Q}$ 
the $\mathbb Q$-linear equivalence of $\mathbb Q$-Cartier 
$\mathbb Q$-divisors. 

We write $\Bs|D|$ the base locus of the linear 
system $|D|$. 
\end{notation}

\begin{ack}
I was partially supported by the Grant-in-Aid for Young Scientists 
(A) $\sharp$20684001 from JSPS. I was 
also supported by the Inamori Foundation. I thank 
the referee for useful comments. 
\end{ack}

\section{Effective base point free theorem}\label{sec2} 

\subsection{Modified base point freeness method after Koll\'ar}\label{2.1}

In this subsection, we slightly generalize Koll\'ar's method in \cite{kollar}. 

\begin{say-a}
Let $(X, \Delta)$ be a log canonical pair and 
$N$ a Cartier divisor on $X$. Let $g:X\to S$ be 
a proper surjective morphism onto a normal variety $S$ with connected 
fibers. 
Let $M$ be a semi-ample $\mathbb Q$-Cartier $\mathbb Q$-divisor 
on $X$. Assume that 
\begin{equation}
N\sim_{\mathbb Q}K_X+\Delta+B+M,
\end{equation} 
where $B$ is an effective $\mathbb Q$-Cartier 
$\mathbb Q$-divisor on $X$ such that 
$\Supp B$ contains no lc centers of $(X, \Delta)$ and 
that $B=g^*(B_S)$, where $B_S$ is an effective 
ample $\mathbb Q$-Cartier $\mathbb Q$-divisor on $S$.  
Let $X\setminus W$ be the largest open set such that 
$(X, \Delta+B)$ is lc. 
Assume that $W\ne \emptyset$, and 
let $Z$ be an irreducible component 
of $W$ such that $\dim g(Z)$ is maximal. We note that 
$g(W)$ is not equal to $S$ since $B=g^*(B_S)$. 
Take a resolution $f:Y\to X$ such that 
the exceptional locus $\Exc(f)$ is a simple normal 
crossing divisor on $Y$,  
and put $h=g\circ f:Y\to S$. We can write  
\begin{equation}
K_Y=f^*(K_X+\Delta)+\sum e_i E_i\ \text{with}\ e_i\geq -1, 
\end{equation} and 
\begin{equation}
f^*B=\sum b_i E_i. 
\end{equation} 
We can assume that $\Supp (f^{-1}_*B\cup 
f^{-1}_*\Delta\cup \sum E_i\cup h^{-1}
(g(Z)))$ and $\Supp (h^{-1}(g(Z)))$ are 
simple normal crossing 
divisors. 
Let $c$ be the largest real number such that $K_X+\Delta +cB$ 
is lc over the generic point of $g(Z)$. 
We note that 
\begin{equation}
K_Y=f^*(K_X+\Delta +cB)+\sum (e_i -cb_i)E_i. 
\end{equation} 
By the assumptions, we know $0<c<1$ and $c\in \mathbb Q$. 
If $cb_i-e_i<0$, then $E_i$ is $f$-exceptional. 
If $cb_i-e_i\geq 1$ and $g(Z)$ is a proper subset of 
$h(E_i)$, then $cb_i-e_i=1$. 
We can write \begin{equation}f^*N\sim _{\mathbb Q} 
K_Y+f^*M+(1-c) f^*B 
+\sum (cb_i-e_i)E_i\end{equation} and 
\begin{equation}
\sum \llcorner cb_i-e_i\lrcorner E_i =F+G_1+G_2-H,
\end{equation} 
where $F$, $G_1$, $G_2$, $H$ are effective 
and without common irreducible 
components such that 
\begin{itemize}
\item the $h$-image of every irreducible component of $F$ is $g(Z)$, 
\item the $h$-image of every irreducible 
component of $G_1$ does 
not contain $g(Z)$, 
\item the $h$-image of every irreducible component of $G_2$ contains 
$g(Z)$ but does not coincide with $g(Z)$, and 
\item $H$ is $f$-exceptional. 
\end{itemize} 
Note that $G_2=\llcorner G_2\lrcorner$ 
is a reduced simple normal crossing 
divisor on $Y$ and that no lc center 
$C$ of $(Y, G_2)$ satisfies $h(C)\subset g(Z)$. 
Here, we used the fact that 
$\Supp (h^{-1}(g(Z)))$ and $\Supp (h^{-1}(g(Z))\cup G_2)$ are 
simple normal crossing divisors on $Y$.  
We put $N'=f^*N+H-G_1$ and consider the short exact sequence 
\begin{equation}
0\to \mathcal O_Y(N'-F)\to \mathcal O_Y(N')\to 
\mathcal O_F(N')\to 0. 
\end{equation} 
Note that $$N'-F\sim _{\mathbb Q}K_Y+f^*M+(1-c)f^*B 
+\sum \{cb_i -e_i\}E_i+G_2. $$
So, the connecting homomorphism 
\begin{equation} 
h_*\mathcal O_F(N')\to R^1h_*\mathcal O_Y(N'-F) 
\end{equation} 
is a zero map since $h(F)=g(Z)$ is a proper subset of $S$ 
and every non-zero local section of 
$R^1h_*\mathcal O_Y(N'-F)$ contains 
$h(C)$ in its support, where $C$ is some stratum of $(Y, G_2)$. 
For the details, see Theorem \ref{ap1}, (a). 
Thus, we know that 
\begin{equation}
0\to h_*\mathcal O_Y(N'-F)\to h_*\mathcal O_Y(N')\to 
h_*\mathcal O_F(N')\to 0\end{equation} 
is exact. 
Moreover, by the vanishing theorem (see Theorem \ref{ap1}, (b)), 
we have 
\begin{equation}
H^1(S, h_*\mathcal O_Y(N'-F))=0. 
\end{equation}
Therefore, 
\begin{equation}\label{100}  
H^0(S, h_*\mathcal O_Y(N'))\to H^0(S, 
h_*\mathcal O_F(N'))
\end{equation} 
is surjective. 
It is easy to see that $F$ is a 
reduced simple normal crossing divisor on 
$Y$. 
We note that no irreducible 
components of $F$ 
appear in $\sum \{cb_i-e_i\}E_i$ 
and that 
\begin{equation}
N'|_{F}\sim _{\mathbb Q} 
K_{F}+(f^*M+(1-c)f^*B)|_{F}
+\sum \{cb_i -e_i\}{E_i}|_{F}+{G_2}|_F. 
\end{equation}
Thus, $h^i(S, h_*\mathcal O_{F}(N'))=0$ for 
all $i>0$ by the vanishing theorem (see 
Theorem \ref{ap1}, (b)).  
Thus, we obtain  
\begin{equation} 
h^0(F, \mathcal O_{F}(N'))=\chi 
(S, h_*\mathcal O_{F}(N')). 
\end{equation}
\end{say-a}
\begin{say-a}
In our application, $M$ will be a variable 
divisor of the form $M_j=M_0+jL$, where $M_0$ is 
a semi-ample $\mathbb Q$-Cartier 
$\mathbb Q$-divisor and $L=g^*L_S$ with 
an ample Cartier divisor $L_S$ on $S$. 
Then we get that 
\begin{equation}
h^0(F, \mathcal O_{F}(N'_0+jf^*L))
=\chi (S, h_*\mathcal O_{F}(N'_0)\otimes 
\mathcal O_S(jL_S))
\end{equation}  
is a polynomial in $j$ for $j\geq 0$, 
where 
\begin{equation}
N'_0=f^*N_0+H-G_1
\end{equation} 
and 
\begin{equation}
N_0\sim _{\mathbb Q} K_X+\Delta +B+M_0.  
\end{equation} 
\end{say-a}

\begin{say-a}\label{2.3} 
Assume that we establish 
$h^0(F, \mathcal O_{F}(N'))\ne 0$. 
By the above surjectivity (\ref{100}), we can lift sections to 
$H^0(Y, \mathcal O_Y(f^*N+H-G_1))$. 
Since $F\not \subset \Supp G_1$, we get a section $s\in 
H^0(Y, \mathcal O_Y(f^*N+H))$ which is not identically 
zero along $F$. 
We know $H^0(Y, \mathcal O_Y(f^*N+H))\simeq 
H^0(X, \mathcal O_X(N))$ because $H$ is $f$-exceptional. 
Thus $s$ descends to a section of $\mathcal O_X(N)$ which does not 
vanish along $Z=f(F)$.  
\end{say-a} 

\subsection{Proof of the main theorem}\label{sub2.2} 

The following lemma, which is the crucial technical 
result needed for Theorem \ref{main}, is essentially the 
same as \cite[2.2.~Lemma]{kollar}. 

\begin{lem-a}\label{lem1}
Let $g:X\to S$ be a proper surjective 
morphism with connected fibers. 
Assume that $X$ is projective, $S$ is normal and $(X, \Delta)$ is lc for some 
effective $\mathbb Q$-divisor $\Delta$. 
Let $D^0_S$ be an ample Cartier divisor 
on $S$ and let $D_S\sim mD^0_S$ for some $m>0$. 
We put $D^0=g^*D^0_S$ and $D=g^*D_S$. 
Assume that $aD^0-(K_X+\Delta)$ is nef and log big for some $a\geq 0$. 
Assume that $|D_S|\ne \emptyset$ and 
that $\Bs |D|$ contains no lc centers 
of $(X, \Delta)$, and let $Z_S\subset \Bs|D_S|$ 
be an irreducible component with 
minimal $k=\codim _S Z_S$. Then, 
with at most $\dim Z_S$ exceptions, 
$Z_S$ is not contained in $\Bs|kD_S+(j+\ulcorner 2a\urcorner
+1)D^0_S|$ for 
$j\geq 0$. 
\end{lem-a}
\begin{proof}
Pick general $B_i\in |D|$ and let 
\begin{equation}
B=\frac{1}{2m}B_0+B_1+\cdots +B_k. 
\end{equation}
Then $B\sim _{\mathbb Q}\frac{1}{2}D^0+kD$, 
$(X, \Delta+B)$ is lc outside $\Bs|D|$ and 
$(X, \Delta+B)$ is not lc at the generic points of $g^{-1}(Z_S)$. 
For the proof, see \cite[(2.1.1) Claim]{kollar}. 
We will apply the method in \ref{2.1} with 
\begin{equation}
N_j=kD+(j+\ulcorner 2a\urcorner+1)D^0
\end{equation}
\begin{equation}
M_0=\ulcorner 2a\urcorner 
D^0-(K_X+\Delta)+\frac{1}{2}D^0, {\text{and}}
\end{equation}
\begin{equation}
M_j=M_0+jD^0. 
\end{equation} 
We note that $M_j$ is semi-ample 
for every $j\geq 0$ 
by Theorem \ref{bpf} since $M_j$ is nef and $M_j-(K_X+\Delta)$ is 
nef and log big. 
The crucial point is to show that 
\begin{equation}
h^0(F, \mathcal O_{F}(N'_j))=\chi 
(S, h_*\mathcal O_{F}(N'_j)) 
\end{equation} 
is not identically zero, where  
\begin{equation}
N'_j=f^*N_j+H-G_1 
\end{equation} 
for every $j$. 
Let $C\subset F$ be a general fiber 
of $F\to h(F)=Z_S$. 
Then 
\begin{equation}
N'_0|_C=(h^*(kD_S+(\ulcorner 2a\urcorner
+1)D^0_S)+H
-G_1)|_C=H|_C. 
\end{equation} 
Hence $h_*\mathcal O_{F}(N'_0)$ is 
not the zero sheaf, and 
\begin{equation}
H^0(F, \mathcal O_{F}(N'_j))
=H^0(S, h_*\mathcal O_{F}(N'_0)\otimes 
\mathcal O_S(jD^0_S))\ne 0
\end{equation} 
for 
$j\gg 1$. 
Therefore, $h^0(F, \mathcal O_{F}(N'_j))$ is a non-zero 
polynomial of degree $\dim Z_S$ in $j$ for $j\geq 0$. Thus it can vanish 
for at most $\dim Z_S$ different values of $j$. This implies 
that 
\begin{equation}
f(F)\not\subset 
\Bs|kD+(j+\ulcorner 2a\urcorner+1)D^0|=g^{-1}
\Bs|kD_S+(j+\ulcorner 2a\urcorner+1)D^0_S|  
\end{equation} 
by \ref{2.3}, with at most $\dim Z_S$ exceptions. 
Therefore, $Z_S=h(F)\not \subset \Bs|kD_S+(j+\ulcorner 2a
\urcorner+1)D^0_S|$. 
This is what we wanted. 
\end{proof}

The next corollary is obvious by Lemma \ref{lem1}. 
For the proof, see \cite[2.3 Corollary]{kollar}. 

\begin{cor-a}\label{25} 
Assume in addition that $m\geq 2a+\dim S$ and 
set $k=\codim _S \Bs|D_S|$. Then 
\begin{equation}
\dim \Bs|(2k+2)D_S|<\dim \Bs |D_S|.  
\end{equation}
\end{cor-a}

\begin{lem-a}\label{key}
We use the same notation as in {\em{Theorem \ref{main}}}. 
Then we can find an effective divisor $D\in |2(\ulcorner a\urcorner+n)L|$ 
such that $D$ contains no lc centers of $(X, \Delta)$. 
\end{lem-a}

\begin{proof}
Let $C$ be an arbitrary lc center of $(X, \Delta)$. 
When $(X, \Delta)$ is kawamata log terminal, 
we put $C=X$. 
We consider the short exact sequence 
\begin{equation}
0\to \mathcal I_C\otimes \mathcal O_X(jL)
\to \mathcal O_X(jL)\to \mathcal O_C(jL)\to 0, 
\end{equation} 
where $\mathcal I_C$ is the defining ideal sheaf of $C$. 
By the vanishing theorem, 
$H^i(X, \mathcal I_C\otimes \mathcal O_X(jL))=
H^i(X, \mathcal O_X(jL))=0$ for all 
$i\geq 1$ and $j\geq a$ 
(see Theorem \ref{ap2}). Therefore, we have 
$H^i(C, \mathcal O_C(jL))=0$ for 
all $i\geq 1$ and $j\geq a$. 
Thus $h^0(C, \mathcal O_C(jL))=\chi 
(C, \mathcal O_C(jL))$ is a non-zero polynomial in $j$ 
since $|mL|$ is base point free for $m\gg 0$ (see Theorem \ref{bpf}). 
On the other hand, the map  
\begin{equation}
H^0(X, \mathcal O_X(jL))\to H^0(C, \mathcal O_C(jL))
\end{equation} 
is surjective for $j\geq a$ since 
$H^1(X, \mathcal I_C\otimes \mathcal O_X(jL))=0$ for 
$j\geq a$ by the vanishing theorem (see 
Theorem \ref{ap2}). 
Thus, with at most $\dim C$ exceptions, 
$C\not \subset \Bs|(\ulcorner a\urcorner+j)L|$ for $j\geq 0$. 
Therefore, we can find an effective 
divisor $D\in |2(\ulcorner a\urcorner+n)L|$ such 
that $D$ contains no lc centers. 
\end{proof}

\begin{proof}[Proof of {\em{Theorem \ref{main}}}] 
By the base point free theorem 
for log canonical pairs (see Theorem \ref{bpf}), 
there exists a positive integer $l$ such 
that $g=\Phi_{|lL|}: X\to S$ is a proper surjective morphism 
onto a normal variety with connected fibers such that 
$L\sim g^*L'$ for some ample Cartier divisor $L'$ on $S$. 
By Lemma \ref{key}, we can find 
$D\in |2(\ulcorner a\urcorner +n)L|$ such 
that $D$ contains no lc centers. 
Then Corollary \ref{25} can be used repeatedly to lower the 
dimension of $\Bs|mL|$. 
This way we obtain that 
$|2^{n+1}(n+1)!(\ulcorner a\urcorner+n)L|$ is base 
point free. 
\end{proof}

We close this section with the following theorem, 
which is the relative version of Theorem \ref{main}. 
We leave the proof for the reader's exercise. 
Of course, we need the relative version of Theorem \ref{ap2} to check 
Theorem \ref{277}. See \cite[Theorem 4.4]{ambro} 
and \cite[Theorem 3.39]{fujino4}. 

\begin{thm-a}\label{277} 
Let $(X, \Delta)$ be a log canonical 
pair with $\dim 
X=n$ and $\pi:X\to V$ a {\em{projective}} surjective morphism. 
Note that $\Delta$ is an effective $\mathbb Q$-divisor on $X$. 
Let $L$ be a $\pi$-nef Cartier divisor on $X$. Assume that 
$aL-(K_X+\Delta)$ is $\pi$-nef and $\pi$-log 
big for some $a\geq 0$. 
Then there exists a positive integer $m=m(n, a)$,  
which only depends on $n$ and $a$, 
such that 
$\mathcal O_X(mL)$ is $\pi$-generated. 
\end{thm-a} 

\section{Appendix:~New cohomological package}\label{sec3}

In this appendix, we quickly review Ambro's formulation 
of Koll\'ar's torsion-free and vanishing theorems. 

\begin{say}
Let $Y$ be a simple normal crossing divisor 
on a smooth 
variety $M$, and let $D$ be a boundary $\mathbb Q$-divisor 
on $M$ such that 
$\Supp (D+Y)$ is simple normal crossing and 
$D$ and $Y$ have no common irreducible components. 
We put $B=D|_Y$ and consider the pair $(Y, B)$. 
Let $\nu:Y^{\nu}\to Y$ be the normalization. 
We put $K_{Y^\nu}+\Theta=\nu^*(K_Y+B)$. 
A {\em{stratum}} of $(Y, B)$ is an irreducible component of $Y$ or 
the image of some lc center of $(Y^\nu, \Theta)$. 
When $Y$ is smooth and $B$ is a boundary $\mathbb Q$-divisor 
on $Y$ such that 
$\Supp B$ is simple normal crossing, we 
put $M=Y\times \mathbb A^1$ and $D=B\times \mathbb A^1$. 
Then $(Y, B)\simeq (Y\times \{0\}, B\times \{0\})$ satisfies 
the above conditions. 
\end{say}
The following theorem is a special 
case of \cite[Theorem 3.2]{ambro}. 

\begin{thm}\label{ap1}
Let $(Y, B)$ be as above. 
Let 
$f:Y\to X$ be a proper morphism and $L$ a Cartier 
divisor on $Y$. 

$(a)$ Assume that $H\sim _{\mathbb Q}L-(K_Y+B)$ is $f$-semi-ample.
Then 
every non-zero local section of $R^qf_*\mathcal O_Y(L)$ contains 
in its support the $f$-image of 
some strata of $(Y, B)$. 

$(b)$ Let $\pi:X\to S$ be a proper morphism, and 
assume that $H\sim _{\mathbb Q}f^*H'$ for 
some $\pi$-ample $\mathbb Q$-Cartier 
$\mathbb Q$-divisor $H'$ on $X$. 
Then, $R^qf_*\mathcal O_Y(L)$ is $\pi_*$-acyclic, that is, 
$R^p\pi_*R^qf_*\mathcal O_Y(L)=0$ for all $p>0$. 
\end{thm}

For the proof of Theorem \ref{ap1}, see \cite[Chapter 2]{fujino4}. 
By the above theorem, we can easily 
obtain the following theorem. 
For the details, 
see \cite[Theorem 4.4]{ambro} 
and \cite[Theorem 3.39]{fujino4}. 

\begin{thm}\label{ap2} 
Let $(X, B)$ be an lc pair. 
Let $C$ be an lc center of $(X, B)$. 
We consider the short exact sequence 
$$
0\to \mathcal I_{C}\to \mathcal O_X\to \mathcal O_{C}\to 0,  
$$ 
where $\mathcal I_{C}$ is the defining ideal sheaf of 
$C$ on $X$. 
Assume that $X$ is projective. 
Let $\mathcal L$ be a line bundle on $X$ such that 
$\mathcal L-(K_X+B)$ is ample. 
Then $H^q(X, \mathcal L)=0$ and 
$H^q(X, \mathcal I_{C}\otimes \mathcal L)=0$ for all 
$q>0$. 
In particular, the restriction map 
$H^0(X, \mathcal L)\to H^0(C, \mathcal L|_{C})$ is 
surjective. 
\end{thm}

A simple proof of Theorem \ref{ap2} can be found in \cite{fuji3} (cf.~\cite[Theorem 4.1]{fuji3}). 
For a systematic treatment on this topic, 
we recommend the reader to see \cite{fujino4}. 

\ifx\undefined\bysame
\newcommand{\bysame|{leavemode\hbox to3em{\hrulefill}\,}
\fi

\end{document}